\documentclass{amsart}

\usepackage{amsmath,amssymb,amsthm}
\usepackage{a4wide}
\usepackage{hyperref}

\usepackage{graphicx}
\usepackage{xypic}
\entrymodifiers={+!!<0pt,\fontdimen22\textfont2>}
\usepackage[all]{xy}

\newtheoremstyle{myremark} 
    {7pt}                    
    {7pt}                    
    {}  	                 
    {}                           
    {\bf}       	         
    {.}                          
    {.5em}                       
    {}  

\theoremstyle{plain}
\newtheorem{lemma}{Lemma}[section]
\newtheorem{theorem}[lemma]{Theorem}

\newtheorem{proposition}[lemma]{Proposition}
\newtheorem{conjecture}[lemma]{Conjecture}

\theoremstyle{myremark}
\newtheorem{remark}[lemma]{Remark}
\newtheorem{example}[lemma]{Example}

\newcommand{\nicef}{\mathcal{F}}
\newcommand{\nicem}{\mathcal{M}}

\newcommand{\ind}{\mathrm{Ind}}

\newcommand{\htpyequiv}{\simeq}

\renewcommand{\subset}{\subseteq}
\newcommand{\susp}{\Sigma}
\newcommand{\lk}{\mathrm{lk}}

\newcommand{\mindeg}{\mathrm{mindeg}}

\newcommand{\redhom}{\widetilde{H}}

\newcommand{\betti}{\mathrm{b}}
\newcommand{\tbetti}{\mathbf{b}}

\newcommand{\abetti}{\beta}

\newcommand{\bipa}{\mathrm{Bip}}

\newcommand{\kkk}{\mathbf{k}}

\newcommand{\neibcomplex}{\mathcal{N}}

\newcommand{\Michal}[1]{}

\begin{document}
\title{Extremal problems related to Betti numbers of flag complexes}

\author[Micha{\l} Adamaszek]{Micha{\l} Adamaszek}
\address{Fachbereich Mathematik, Universit\"at Bremen
      \newline Bibliothekstr. 1, 28359 Bremen, Germany}
\email{aszek@mimuw.edu.pl}
\thanks{This research was partially carried out when the author was a member of the Centre for Discrete Mathematics and its Applications (DIMAP) and the Mathematics Institute of the University of Warwick, Coventry, UK. Research supported by EPSRC award EP/D063191/1 and a DFG grant.}


\keywords{Flag complex, Betti numbers, Independence complex, Euler characteristic}
\subjclass[2010]{05E45, 55U10}

\begin{abstract}
We study the problem of maximizing Betti numbers of simplicial complexes. We prove an upper bound of $1.32^n$ for the sum of Betti numbers of any $n$-vertex flag complex and $1.25^n$ for the independence complex of a triangle-free graph. These findings imply upper bounds for the Betti numbers of various related classes of spaces, including the neighbourhood complex of a graph. We also make some related observations.
\end{abstract}
\maketitle

\section{Introduction}
\label{sect:intro}

There is a number of topological and algebraic invariants associated to graphs and simplicial complexes. The starting point for this investigation is the following kind of extremal problem: given a bound on the size of the combinatorial input, how large values can those invariants take?

For example, if a simplicial complex has at most $n$ vertices, then it has at most $2^n$ faces, and trivially its homology groups can have total dimension at most $2^n$. This is asymptotically optimal. The $k$-skeleton of the $n$-simplex is known to be homotopy equivalent to the wedge of ${n \choose k+1}$ spheres, hence for $k\approx n/2$ its homology has dimension approximately $\frac{2^n}{\sqrt{n}}$. This construction is optimal by \cite{BK}.

Considerably better upper bounds can be obtained by considering restricted families of complexes. In this work we concentrate on flag complexes. A simplicial complex $K$ is called \emph{flag} if the set of faces of $K$ coincides with the set of cliques in the $1$-skeleton of $K$, hence $K$ is also called a \emph{clique complex}. Clique complexes appear throughout geometry, topology and combinatorics, some notable examples being order complexes of posets and Vietoris-Rips complexes of discrete metric spaces.

From the point of view of our applications it will be convenient to take the following, equivalent approach to flag complexes. The \emph{independence complex} $\ind(G)$ of a graph $G$ is a simplicial complex whose faces are the independent sets of $G$ (the sets $W\subset V(G)$ for which the induced subgraph $G[W]$ has no edges). Clearly the family of independence complexes of graphs is the same as the family of flag complexes as an independent set in $G$ is a clique in the graph-theoretic complement $\overline{G}$.

We fix once and for all a field $\kkk$. For a finite type topological space $X$ let $\betti(X)=\sum_{i}\dim_\kkk\redhom_i(X;\kkk)$ denote the \emph{total Betti number} of $X$. In Section~\ref{section:proof1} we will define constants 
$$\Theta=4^{1/5}\approx 1.32 \quad\textrm{and}\quad \Gamma\approx 1.25$$
for which we have the next result.

\begin{theorem}
\label{thm:maintheorem1}
For any $n$-vertex graph $G$ we have 
\begin{equation*}
\betti(\ind(G)) \leq \Theta^n.
\end{equation*}
If $G$ is triangle-free then 
\begin{equation*}
\betti(\ind(G)) \leq \Gamma^n.
\end{equation*}
\end{theorem}

Let us make a few comments. The first inequality can also be deduced from more general results of \cite{KozHulls}, although the above formulation seems far from being ``well-known''. For the smaller class of order complexes one can use an even earlier result of \cite{Ziegler}. Theorem~\ref{thm:maintheorem1} provides a benchmark for studying families of graphs $G$ for which $\betti(\ind(G))$ is exponential in the number of vertices $n$. Such graphs have appeared recently in mathematical physics, see for example \cite{EngWitten,HuSch1,HuSch2}. Independence complexes of bipartite and triangle-free graphs have also received some attention \cite{Bar,JonBip}, as well as those of other classes of graphs with forbidden induced subgraphs \cite{Eng2}. It should also be noted that the above are the extremal values, only approached by tailor-made constructions. The behaviour of Betti numbers for random or geometric flag complexes is a lot more tame, see \cite{Kahle1,Goff}.

The second part of Theorem~\ref{thm:maintheorem1} has various consequences listed below. Recall that the \emph{neighbourhood complex} $\neibcomplex(G)$ of a graph $G$ is a simplicial complex whose vertices are the non-isolated vertices of $G$ and whose faces are the vertex subsets which have a common neighbour in $G$. It is a well-known construction with classical applications to the theory of chromatic numbers of graphs \cite{Lovasz}.

\begin{theorem}
\label{thm:maintheorem3}
We have the following upper bounds.
\begin{itemize}
\item[a)] If $K$ is any simplicial complex with $n$ vertices and $m$ maximal faces then
\begin{equation*}
\betti(K)\leq\Gamma^{n+m}.
\end{equation*}
\item[b)] If $G$ is a graph with $n$ vertices then its neighbourhood complex $\neibcomplex(G)$ satisfies
\begin{equation*}
\betti(\neibcomplex(G))\leq(\Gamma^2)^{n}.
\end{equation*}
\item[c)] If $K$ is any simplicial complex with $n$ vertices and $m$ minimal non-faces then
\begin{equation*}
\betti(K)\leq\Gamma^{n+m}.
\end{equation*}
\end{itemize}
\end{theorem}
Clearly parts a),c) give useful bounds (better than $2^n$) only when $m$ is sufficiently small, approx. $m\leq 2.1n$.

\medskip
The paper is laid out as follows. Section~\ref{section:prelim} contains the necessary background and notation. In Section~\ref{section:proof1} we give a proof of Theorem~\ref{thm:maintheorem1} based on the inequality~\eqref{eq:bettiinequality}. As an immediate application we use Hochster's formula to derive an upper bound for the sum of algebraic Betti numbers $\abetti(G)$ of flag complexes and edge ideals.

Theorem~\ref{thm:maintheorem3} is proved in Section~\ref{section:proof3} using the correspondence between arbitrary simplicial complexes and independence complexes of bipartite graphs.

In Section~\ref{section:nomissing} we investigate complexes without missing $d$-faces, which are a generalization of flag complexes. We use Alexander duality to show an analogue, albeit much weaker, of Theorem~\ref{thm:maintheorem1} for such complexes. As an aside, we show how the same methods give bounds on the homological dimension in those classes of complexes. 

Unfortunately, of all the upper bounds we present, only the first one in Theorem~\ref{thm:maintheorem1} is known to be tight. In Section~\ref{section:lower} we will construct examples exhibiting the best lower bounds we were able to find. It is likely that some of them are in fact optimal. They are summarized in Table~\ref{table1}.

\begin{table}
\begin{tabular}{l|r|r}
quantity & best construction & upper bound \\
\hline
$\betti(\ind(G))$ & (Ex.\ref{ex:bigBetti}) $\ \ \ 1.320^n$ & (Thm.\ref{thm:maintheorem1}) $\ \ \ 1.320^n$ \\
$\betti(\ind(G))$, $G$ triangle-free & (Ex.\ref{ex:bigBettiBip}) $\ \ \ 1.160^n$ & (Thm.\ref{thm:maintheorem1}) $\ \ \ 1.250^n$ \\
$\betti(\neibcomplex(G))$ & (Ex.\ref{ex:neibcomplex}) $\ \ \ 1.316^n$ & (Thm.\ref{thm:maintheorem3}) $\ \ \ 1.562^n$ \\
$\abetti(G)$ & (Ex.\ref{ex:bigAlgBetti}) $\ \ \ 2.299^n$ & (Prop.\ref{thm:maintheorem2}) $\ \ \ 2.320^n$ \\
$\abetti(G)$, $G$ triangle-free & (Ex.\ref{ex:bigAlgiTriangleFree}) $\ \ \ 2.070^n$ & (Prop.\ref{thm:maintheorem2}) $\ \ \ 2.250^n$ 
\end{tabular}
\vskip0.2cm
\caption{\label{table1} Lower and upper bounds appearing in this paper, with $n$ denoting the number of vertices of $G$.}
\end{table}

\section{Preliminaries}
\label{section:prelim}

We first introduce some notation and prove basic results.

\subsection*{Graphs.}
All graphs are finite, undirected and without multiple edges or loops. If $v\in V(G)$ then $N_G(v)$ is the set of neighbours of $v$ and $N_G[v]=N_G(v)\cup\{v\}$. The degree of $v$ is $\deg_Gv=|N_G(v)|$ and $\mindeg(G)$ is the smallest degree of a vertex of $G$. If $W\subset V(G)$ then $G[W]$ is the induced subgraph of $G$ with vertex set $W$ and $G\setminus W$ is shorthand for the more correct $G[V(G)\setminus W]$.

By $G\sqcup H$ we denote the disjoint union of graphs $G$ and $H$. The notation $\bigsqcup^s G$ or simply $sG$ stands for the disjoint union of $s$ copies of $G$. By $G\oplus H$ we denote the graph obtained from $G\sqcup H$ by adding all edges between $V(G)$ and $V(H)$. In the context of graphs we denote by $\emptyset$ the unique graph with $V(G)=\emptyset$. For $s\geq 1$ let $K_s$ denote the complete graph with $s$ vertices.

\subsection*{Simplicial topology.}
If $K$ and $L$ are simplicial complexes then $K\ast L$ is their join, i.e. the complex with faces of the form $\sigma\sqcup\tau$ for $\sigma\in K$ and $\tau\in L$. If $G$ and $H$ are graphs then 
\begin{equation}\label{eq:joinind}\ind(G\sqcup H)=\ind(G)\ast\ind(H).\end{equation}
By $\susp\, K=S^0\ast K$ we denote the (unreduced) suspension of $K$, where $S^0$ is the $0$-sphere, that is the simplicial complex with two isolated vertices. In the context of topological spaces $\emptyset$ denotes the empty space. We have $\ind(\emptyset)=\emptyset$.

If $\nicef\subset 2^V$ is a family of subsets of some ground set $V$, then the \emph{simplicial complex generated by} $\nicef$ is $\{A~|~A\subset F\ \textrm{for some}\ F\in\nicef \}$. If $K$ is a simplicial complex then a \emph{non-face} of $K$ is a subset of $V(K)$ which is not a face of $K$.

The $i$-th Betti number of a space $X$ is $\betti_i(X)=\dim_\kkk\redhom_i(X;\kkk)$, where $\redhom_i$ denotes the $i$-th reduced homology group. The \emph{total Betti number} is $\betti(X)=\sum_i\betti_i(X)$. We have $\betti_i(\susp\,X)=\betti_{i-1}(X)$, hence $\betti(\susp\,X)=\betti(X)$.

If $K$ is a simplicial complex then $K^*$ is its \emph{Alexander dual}, i.e. the simplicial complex with vertex set $V(K)$ and with faces $\{\sigma\subset V(K)~:~V(K)\setminus\sigma\not\in K\}$. By Alexander duality \cite{BjoTancer} we have $\betti_i(K)=\betti_{n-i-3}(K^*)$ where $n$ is the number of vertices in $K$ \footnote{Alexander duality relates homology of $K$ with the cohomology of $K^*$, but the latter is isomorphic to homology since we decided to work over a field. This is just for convenience --- the proofs go through over any coefficient ring after adjusting for the use of cohomology in place of homology at certain points.}. It follows that $\betti(K)=\betti(K^*)$. 

If $v$ is a vertex of a simplicial complex $K$ then we define the \emph{link} of $v$ as $\lk_Kv=\{\tau\in K~:~ v\not\in\tau,\ \tau\cup\{v\}\in K\}$. It is a standard observation that there is a cofibre sequence
$$\lk_Kv\to K\setminus v\to K$$
and therefore
\begin{equation}
\label{eq:bettiinequality}
\betti(K)\leq \betti(K\setminus v)+\betti(\lk_Kv).
\end{equation}

If $K=\ind(G)$ then we have $K\setminus v=\ind(G\setminus v)$ and $\lk_Kv=\ind(G\setminus N_G[v])$.

We will use the shorthand notation
$$\tbetti(G):=\betti(\ind(G))=\sum_{i}\betti_i(\ind(G)).$$

Here are some standard properties of the function $\tbetti(G)$.
\begin{lemma}
\label{lem:betti}
We have the following properties.
\begin{itemize}
\item[a)] $\tbetti(G\sqcup H)=\tbetti(G)\tbetti(H)$,
\item[b)] $\tbetti(G)=0$ if $G$ has an isolated vertex,
\item[c)] for any $v\in V(G)$ $$\tbetti(G)\leq \tbetti(G\setminus v)+\tbetti(G\setminus N_G[v]),$$ 
\item[d)] $\tbetti(\emptyset)=1$.
\end{itemize}
\end{lemma}
\begin{proof}
The reduced homology of the join is given, over a field, by the formula (eg.\cite[Lemma 2.1]{Milnor})
$$\redhom_k(X\ast Y)=\bigoplus_{\substack{i,j\geq -1\\ i+j=k-1}} \redhom_i(X)\otimes \redhom_j(Y), \qquad k\geq -1$$
That implies $\betti(X\ast Y)=\betti(X)\betti(Y)$ and in consequence we get parts a) and b). Part b) also follows since $\ind(G)$ is a cone whose apex is the isolated vertex. 

Part c) is just \eqref{eq:bettiinequality} applied to $K=\ind(G)$.

Part d) holds because the empty space has a single non-trivial reduced homology group in degree $-1$, that is $\betti_{-1}(\emptyset)=1$ and $\betti_i(\emptyset)=0$ for $i\geq 0$.
\end{proof}

\begin{remark}
\label{remark:euler}
There are other functions which satisfy conditions a)-d) of the above lemma. One of them is $|\widetilde{\chi}(\ind(G))|$, where $\widetilde{\chi}$ is the reduced Euler characteristic of a space. The quantity $\widetilde{\chi}(\ind(G))$ is the evaluation at $-1$ of the independence polynomial of $G$, and as such has been studied before, see eg. \cite{LeMa, LeMa2}. Moreover, in all of the lower bound constructions for $\tbetti(G)$ in Section~\ref{section:lower} we will use graphs $G$ for which the spaces $\ind(G)$ are, up to homotopy, wedges of spheres of the same dimension. It follows that those examples are also extremal for $|\widetilde{\chi}(\ind(G))|$, with the same extremal values.
\end{remark}

\section{Proof of Theorem~\ref{thm:maintheorem1}}
\label{section:proof1}

\subsection*{Definitions of $\Theta$ and $\Gamma$.} We begin by defining the required constants. Let $\Theta_d=d^{1/(d+1)}$ for $d\geq 1$. An elementary check shows that $\Theta_4\geq \Theta_d$ for all $d\geq 1$ and we define 
\begin{equation}\label{eq:defntheta}\Theta=\Theta_4=4^{1/5}\approx 1.320.\end{equation}

Next, consider, for any $d\geq 1$, the function
$$f_d(x)=x^{-(d+1)}+x^{-(d+2)}+\cdots+x^{-2d}.$$
It is decreasing for $x\in[1,2]$ with $f_d(1)\geq 1$ and $f_d(2)<1$, so there is a unique $\Gamma_d\in[1,2]$ for which $f_d(\Gamma_d)=1$. Again, an easy check reveals that $\Gamma_3\geq \Gamma_d$ for all $d\geq 1$, and we set
\begin{equation}\label{eq:defngamma}\Gamma=\Gamma_3\approx 1.250.\end{equation}

Now we can prove our main result.
\begin{theorem}[Theorem~\ref{thm:maintheorem1}]
For any graph $G$ with at most $n$ vertices we have $\tbetti(G) \leq \Theta^n$. If $G$ is triangle-free then $\tbetti(G) \leq \Gamma^n$.
\end{theorem}
\begin{proof}
We proceed by induction. The result holds for $n=0$ by Lemma~\ref{lem:betti}.d). 

Consider a graph $G$ with $n$ vertices. Let $d=\mindeg(G)$. If $d=0$ then $G$ has an isolated vertex and we use Lemma~\ref{lem:betti}.b). Suppose that $d\geq 1$ and choose any vertex $v$ of degree exactly $d$.  Denote by $v_1,\ldots, v_d$ the neighbours of $v$ in $G$. Next, let $G_i=G\setminus \{v_1,\ldots,v_i\}$ for $i=1,\ldots,d$ and $G_0=G$.

Consecutive application of Lemma~\ref{lem:betti}.c) gives the following bound:
\begin{eqnarray*}
\tbetti(G)=\tbetti(G_0)&\leq &\tbetti(G_1)+\tbetti(G_0\setminus N_{G_0}[v_1]) \\
&\leq &\tbetti(G_2)+\tbetti(G_1\setminus N_{G_1}[v_2])+\tbetti(G_0\setminus N_{G_0}[v_1]) \\
&\cdots&\\
&\leq&\tbetti(G_{d})+\sum_{i=0}^{d-1}\tbetti(G_i\setminus N_{G_i}[v_{i+1}]).
\end{eqnarray*}
However, $G_d$ has $v$ as an isolated vertex, so $\tbetti(G_d)=0$. Eventually we obtain
\begin{equation}
\label{eq:sumbound}
\tbetti(G)\leq\sum_{i=0}^{d-1}\tbetti(G_i\setminus N_{G_i}[v_{i+1}]).
\end{equation}
Since each vertex of $G$ has degree at least $d$, we have $|V(G_i\setminus N_{G_i}[v_{i+1}])|\leq n-d-1$ for all $i=0,\ldots,d-1$. Plugging this into \eqref{eq:sumbound} and using the inductive assumption we obtain
$$\tbetti(G)\leq d\cdot\Theta^{n-d-1}=\Theta^n\cdot\frac{d}{\Theta^{d+1}}\leq \Theta^n\cdot\frac{d}{\Theta_d^{d+1}}=\Theta^n.$$
This completes the proof of the first part.

If, in addition, $G$ is triangle-free then $N_{G}(v_{i+1})\cap N_G(v)=\emptyset$ for all $i=0,\ldots,d-1$ and so $\deg_{G_i}v_{i+1}\geq d$. It follows that $|V(G_i\setminus N_{G_i}[v_{i+1}])|\leq n-d-i-1$. All of the graphs $G_i\setminus N_{G_i}[v_{i+1}]$ are again triangle-free, so by induction we get
\begin{eqnarray*}
\tbetti(G)&\leq & \sum_{i=0}^{d-1}\Gamma^{n-i-d-1}\\
&=&\Gamma^n\sum_{i=0}^{d-1}\Gamma^{-(d+i+1)}=\Gamma^nf_d(\Gamma)\leq\Gamma^nf_d(\Gamma_d)=\Gamma^n.
\end{eqnarray*} 
The proof of the theorem is complete.
\end{proof}

We end this section with an application to the algebraic Betti numbers. Let $\abetti(G)=\sum_{i,\sigma}\abetti_{i,\sigma}(\ind(G))$ be the sum of the finely graded Betti numbers of the Stanley-Reisner ring of the complex $\ind(G)$, that is the total dimension of the minimal resolution of that ring as 	a module over the polynomial ring $\kkk[V(G)]$. We refer to \cite[Chapters 1,5]{MS} for introduction to commutative algebra and resolutions of monomial ideals. The classical Hochster's formula gives
\begin{equation}
\label{eq:hoch}
\abetti(G)=\sum_{W\subseteq V(G)} \tbetti(G[W])
\end{equation}
which, for our purpose, can be just as well taken for the definition of $\abetti(G)$. Note that Hochster's formula together with Lemma~\ref{lem:betti}.a) gives
\begin{equation}
\label{eq:abettiprod}
\abetti(G\sqcup H)=\abetti(G)\abetti(H).
\end{equation}
We immediately get the following bounds.

\begin{proposition}
\label{thm:maintheorem2}
For any $n$-vertex graph $G$ we have
\begin{equation*}
\abetti(G) \leq (\Theta+1)^n.
\end{equation*}
If $G$ is triangle-free then 
\begin{equation*}
\abetti(G) \leq (\Gamma+1)^n.
\end{equation*}
\end{proposition}
\begin{proof}
Use \eqref{eq:hoch} and Theorem~\ref{thm:maintheorem1} to get:
\begin{eqnarray*}
\abetti(G)&=&\sum_k\sum_{W\in {V(G)\choose k}}\tbetti(G[W])\\
&\leq&\sum_k{n\choose k}\Theta^k=(\Theta+1)^n.
\end{eqnarray*}
If $G$ is triangle-free then so are all its induced subgraphs and the proof goes through with $\Gamma$ in place of $\Theta$.
\end{proof}

\section{Further results}
\label{section:proof3}
The main purpose of this Section is to prove Theorem~\ref{thm:maintheorem3}. We first recall a construction which relates arbitrary simplicial complexes to independence complexes of bipartite graphs.

Suppose $K$ is a simplicial complex with $n$ vertices $v_1,\ldots,v_n$ and $m$ maximal faces $F_1,\ldots,F_m$. We construct a bipartite graph $\bipa(K)$ as follows. The two parts of $V(\bipa(K))$ are $\{v_1,\ldots,v_n\}$ and $\{F_1,\ldots,F_m\}$. There is an edge $v_iF_j$ in $\bipa(K)$ if and only if $v_i\not\in F_j$ in $K$.

The following is a result of \cite[Thm.3.7]{Bar} and \cite[Sect. 3]{JonBip}.
\begin{lemma}
\label{fact:bipsim}
For a nonempty simplicial complex $K$ we have a homotopy equivalence
$$\ind(\bipa(K))\htpyequiv\susp\,K.$$
\end{lemma}

Equipped with this tool we proceed with the proof of Theorem~\ref{thm:maintheorem3}.

\begin{proof}[Proof of Theorem~\ref{thm:maintheorem3}]
a) If $K$ has $n$ vertices and $m$ maximal faces then the bipartite graph $\bipa(K)$ has $n+m$ vertices. Using Lemma~\ref{fact:bipsim} and Theorem~\ref{thm:maintheorem1} we get
$$\betti(K)=\tbetti(\bipa(K))\leq\Gamma^{n+m}.$$

b) The complex $\neibcomplex(G)$ is generated by faces of the form $N_G(v)$ for all non-isolated vertices $v$ of $G$. In particular, $\neibcomplex(G)$ has at most $n$ vertices and at most $n$ maximal faces, so part a) applies.

c) The maximal faces of the Alexander dual $K^*$ are of the form $V(K)\setminus F$ for the minimal non-faces $F$ of $K$. It follows that $\betti(K)=\betti(K^*)\leq \Gamma^{n+m}$ where the first equality is Alexander duality and the second inequality follows from part a) applied to $K^*$.

\end{proof}

For future reference let us also note the following fact. It generalizes \cite[Prop.8]{CsoOsz}, where it was proved when $H=K_s$.

\begin{lemma}
\label{lem:neibjoin}
For any nonempty graphs $G$ and $H$ there is a homotopy equivalence
$$\neibcomplex(G\oplus H)\htpyequiv\susp\,\big(\neibcomplex(G)\ast\neibcomplex(H)).$$
In particular
$$\betti(\neibcomplex(G\oplus H))=\betti(\neibcomplex(G))\cdot\betti(\neibcomplex(H)).$$
\end{lemma}
\begin{proof}
Let $X, Y$ be two subcomplexes of $\neibcomplex(G\oplus H)$ defined as
\begin{eqnarray*}
X&=&\{\sigma\sqcup\tau~:~\sigma\in\neibcomplex(G), \tau\subset V(H)\},\\
Y&=&\{\sigma\sqcup\tau~:~\sigma\subset V(G), \tau\in\neibcomplex(H)\}.
\end{eqnarray*}
We see that $X\cup Y=\neibcomplex(G\oplus H)$. Indeed, suppose that $\sigma\sqcup\tau\subset V(G)\cup V(H)$ is a set of vertices all of which have a common neighbour in $G\oplus H$. Assuming, without loss of generality, that the common neighbour belongs to $V(G)$, we get $\sigma\in\neibcomplex(G)$. On the other hand, every set $\sigma\sqcup\tau$ with $\sigma\in\neibcomplex(G)$ and $\tau\subset V(H)$ has a common neighbour in $V(G)\subset V(G\oplus H)$.

If $\Delta(A)$ denotes the full simplex with vertex set $A$, then we have
$$X=\neibcomplex(G)\ast\Delta(V(H)), \quad Y=\Delta(V(G))\ast\neibcomplex(H)$$
hence $X$ and $Y$ are contractible. Finally
$$X\cap Y=\{\sigma\sqcup\tau~:~\sigma\in\neibcomplex(G), \tau\in\neibcomplex(H)\}=\neibcomplex(G)\ast\neibcomplex(H).$$
The proof is completed by using the standard homotopy equivalence $X\cup Y\htpyequiv\susp(X\cap Y)$ for contractible subcomplexes $X, Y$ of $X\cup Y$.

\end{proof}

\section{Complexes without missing $d$-faces}
\label{section:nomissing}

Let $\nicef(n,d)$ be the family of simplicial complexes $K$ on $n$ vertices such that every minimal non-face of $K$ has cardinality at most $d$. Set $\nicef_d=\bigcup_n\nicef(n,d)$. These are usually referred to as complexes \emph{without missing faces of dimension at least $d$}, or briefly complexes \emph{without missing $d$-faces}. The family $\nicef_2$ is precisely that of  flag complexes. Since $\nicef(n,d)\subseteq \nicef(n,d+1)$ one might expect a whole hierarchy of inequalities generalizing that of Theorem~\ref{thm:maintheorem1}.

By $\nicem(n,d)$ we denote the family of simplicial complexes $K$ on $n$ vertices such that every maximal face $F$ of $K$ satisfies $|F| \geq n-d$ and we set $\nicem_d=\bigcup_n\nicem(n,d)$. Note that the classes $\nicef(n,d)$ and $\nicem(n,d)$ are Alexander dual to each other:
$$K\in \nicef(n,d) \iff K^*\in\nicem(n,d).$$
We will need the following simple observation.

\begin{lemma}
\label{lem:mclosed}
The class $\nicem_d$ is closed under taking links and vertex deletions.
\end{lemma}
\begin{proof}
Suppose $K\in\nicem(n,d)$ and let $v\in V(K)$. Every maximal face of $K\setminus v$ is of the form $F\setminus v$ for a maximal face $F$ of $K$, hence it has cardinality at least $n-d-1$ which proves $K\setminus v\in \nicem(n-1,d)$.

Next, suppose that $\lk_K v$ has $n'\leq n-1$ vertices and let $F'$ be any maximal face in $\lk_K v$. Then $F'\cup\{v\}$ is a maximal face in $K$ and therefore
$$|F'|+1\geq n-d\geq n'+1-d$$
so $|F'|\geq n'-d$ as required.
\end{proof}

We can now formulate the analogue of Theorem~\ref{thm:maintheorem1} for the class $\nicef_d$. Define $\theta_d\in [1,2]$ as the unique solution of the equation
$$x^d=\sum_{i=0}^{d-1}x^i.$$
\begin{theorem}
\label{thm:missingfacesbound}
If $K\in \nicef(n,d)$ or $K\in\nicem(n,d)$ then $\betti(K)\leq (\theta_d)^n$.
\end{theorem}
\begin{proof}
By the previous remarks about Alexander duality it suffices to consider $K\in\nicem(n,d)$. We prove the inequality by induction on $n$, with the case $n=0$ being obvious. Let $n\geq 1$. For a complex $K\in \nicem(n,d)$ let $F\neq \emptyset$ be any maximal face and let $v_1,\ldots,v_s$ be the vertices in $V(K)\setminus F$. By assumption we have $s\leq d$. Let $K_i=K\setminus\{v_1,\ldots,v_i\}$ and $K_0=K$. 

Using repeatedly the inequality \eqref{eq:bettiinequality} we obtain
$$\betti(K)\leq \sum_{i=0}^{s-1}\betti(\lk_{K_i}v_{i+1})+\betti(K\setminus \{v_1,\ldots,v_s\}).$$
The last term is $\betti(F)=0$ since $F$ is a nonempty simplex. For every $i$ we have $|V(K_i)|=n-i$ hence $|V(\lk_{K_i}v_{i+1})|\leq n-i-1$. By the induction hypothesis and Lemma~\ref{lem:mclosed}:
\begin{eqnarray*}
\betti(K)&\leq& \sum_{i=0}^{s-1}\theta_d^{n-i-1}\leq\sum_{i=0}^{d-1}\theta_d^{n-i-1}=\theta_d^n.
\end{eqnarray*}
\end{proof}

For $d=2$ Theorem~\ref{thm:missingfacesbound} gives suboptimal bounds since $\theta_2=\frac{1+\sqrt{5}}{2}\approx 1.61$ while $\Theta\approx 1.32$. We suspect that the actual upper bound for $\betti(K)$ for $K\in\nicef_d$ is given by the complexes described in Example~\ref{ex:missingfaces}.

\begin{remark}
The proof of Theorem~\ref{thm:maintheorem1} can be easily adapted to show the classical result of Moon and Moser \cite{MM}, which in our language states that an $n$-vertex flag complex $K$ has at most $(3^{1/3})^n$ maximal faces (the bound is optimal). The analogous problem of bounding the number of maximal faces for complexes in $\nicef(n,d)$, for fixed $d$, is not well understood and seems hard --- to our best knowledge the only result in this direction is the paper \cite{Lonc} with $d=3$. This can be an indicator that also finding the optimal bound in Theorem~\ref{thm:missingfacesbound} is not easy.
\end{remark}

\medskip
As an aside, we apply the above framework to give a homotopy-theoretic proof of the following fact about the location of Betti numbers of complexes without missing $d$-faces.
\begin{proposition}
\label{prop:nohighhomology}
Suppose $K\in\nicef(n,d)$. Then $\redhom_i(K)=0$ for $i>n\cdot\frac{d-1}{d}-1$.
\end{proposition}
For instance, when $d=2$ it says that a flag complex has no homology in dimensions $i>\frac{n}{2}-1$ or, in other words, that a flag complex requires at least $2i+2$ vertices to have nontrivial $i$-th homology group. This is a well-known and exploited fact.

\begin{proof}
Let $j=n-i-1$. The inequality $i>n\cdot\frac{d-1}{d}-1$ is equivalent to $j<\frac{n}{d}$. Since $K^*\in\nicem(n,d)$ the intersection of any $j$ maximal faces of $K^*$ contains at least $n-dj$ vertices, in particular it is nonempty as $n-dj>0$. It means that the nerve of the family of maximal faces of $K^*$ has full $(j-1)$-dimensional skeleton and therefore it has trivial (co)homology in dimension $j-2$. By the nerve theorem the same holds for $K^*$. Then, by Alexander duality, $K$ has trivial homology in dimension $n-(j-2)-3=i$.
\end{proof}

\section{Lower bounds}
\label{section:lower}
In this section we describe various families of graphs and complexes with exponentially large Betti numbers, complementing the main upper bounds. All these examples arise from a carefully chosen set small graphs by taking multiple disjoint unions or joins. First observe that the space $\ind(K_s)$ is a disjoint union of $s$ points and therefore $\tbetti(K_s)=s-1$.

\begin{example}
\label{ex:bigBetti}
Suppose $n$ is divisible by $s$ and let $G_n=\bigsqcup^{n/s}K_s$. By Lemma~\ref{lem:betti}.a) we have
$$\tbetti(G_n)=\tbetti(K_s)^{n/s}=((s-1)^{1/s})^n.$$
The quantity $(s-1)^{1/s}$ is maximized for $s=5$, so when $n$ is divisible by $5$ one gets an $n$-vertex graph $G_n$ with $\tbetti(G_n)=\Theta^n$. This shows that the first bound in Theorem~\ref{thm:maintheorem1} is tight.

The proof of Theorem~\ref{thm:maintheorem1} can be adapted to show that disjoint unions of $K_5$'s are the unique extremal graphs for this problem (if $n$ is not divisible by $5$ one must adjust the size of one or two cliques). We omit the details.
\end{example}

\begin{example}
\label{ex:bigBettiBip}
We will construct our best example of a bipartite graph $G_n$ with large $\tbetti(G_n)$. To make the construction more transparent we go via the correspondence described in Section~\ref{section:proof3}.

Let $K$ be the two-dimensional simplicial complex with vertex set $\{1,\ldots,7\}$ and with $7$ maximal faces
$$\{1,4,5\},\,\{1,3,6\},\,\{1,2,7\},\,\{2,3,5\},\,\{2,4,6\},\,\{3,4,7\},\,\{5,6,7\}.$$
Note that this is precisely the Steiner triple system associated to the Fano plane. The $1$-skeleton of $K$ is the complete graph $K_7$ and every edge belongs to exactly one triangle. It follows that $K$ collapses to a graph with $7$ vertices and ${7 \choose 2}-7=14$ edges. In particular, it is homotopy equivalent to a wedge of $14-7+1=8$ circles, hence  $\betti(K)=8$.

The bipartite graph $\bipa(K)$ has $14$ vertices ($7$ in each part) and $\tbetti(\bipa(K))=\betti(K)=8$ by Lemma~\ref{fact:bipsim}.

If $n$ is divisible by $14$ then the graph $G_n=\bigsqcup^{n/14}\bipa(K)$ is a bipartite graph with 
$$\tbetti(G_n)=8^{n/14}=(8^{1/14})^n\approx 1.160^n.$$
We were not able to improve on this example in the class of all triangle-free graphs --- see Conjecture~\ref{conj:biptri}.

\end{example}

\medskip
For the examples involving the algebraic Betti numbers we need the following auxiliary results.

\begin{lemma}
\label{lem:abetticlique}
For any integer $s\geq 1$ we have $\abetti(K_s)=2^{s-1}(s-2)+2.$
\end{lemma}
\begin{proof}
The clique $K_s$ has ${s\choose i}$ induced subgraphs isomorphic to $K_i$. Taking into account also the empty subgraph, \eqref{eq:hoch} yields:
$$\abetti(K_s)=1+\sum_{i=1}^s{s\choose i}(i-1).$$
That this sum equals $2^{s-1}(s-2)+2$ is left as an easy exercise.
\end{proof}

Let $\widetilde{K_{s,s}}=K_{s,s}-M$ denote the complete bipartite graph $K_{s,s}$ from which a perfect matching $M$ was removed (i.e. the unique $(s-1)$-regular bipartite graph with two parts of size $s$).

\begin{lemma}
\label{lem:abettibip}
For any integer $s\geq 1$ we have $\abetti(\widetilde{K_{s,s}})=4^{s-1}(s-4)+2\cdot 3^s-2^{s+1}+2.$
\end{lemma}
\begin{proof}
The complex $\ind(\widetilde{K_{s,s}})$ consists of two $(s-1)$-simplices connected by $s$ line segments (the edges of $M$). An induced subcomplex $\ind(G[W])$ has the homotopy type of
\begin{itemize}
\item $\emptyset$, if $W=\emptyset$,
\item a point, if $W\neq\emptyset$ is contained in one part of $V(\widetilde{K_{s,s}})$,
\item $S^0$, if $W$ contains at least one vertex in each part, but contains no edge of $M$,
\item the wedge $\bigvee^{q-1}S^1$, if $W$ contains exactly $q$ edges of $M$, $q=1,\ldots,s$.
\end{itemize} 
Counting these types of subgraphs yields, via \eqref{eq:hoch}
$$\abetti(\widetilde{K_{s,s}})=1+(3^s-2^{s+1}+1)+\sum_{q=1}^s{s \choose q}3^{s-q}(q-1).$$
Calculating this sum is an easy exercise.
\end{proof}

\begin{example}
\label{ex:bigAlgBetti}
Suppose $n$ is divisible by $s$ and let $G_n=\bigsqcup^{n/s}K_s$. By \eqref{eq:abettiprod} and Lemma \ref{lem:abetticlique} we have
$$\abetti(G_n)=\abetti(K_s)^{n/s}=((2^{s-1}(s-2)+2)^{1/s})^n.$$
One checks that $(2^{s-1}(s-2)+2)^{1/s}$ is maximized for $s=9$, so when $n$ is divisible by $9$ we have
$$\abetti(G_n)=((1794)^{1/9})^n\approx 2.299^n.$$
\end{example}

\begin{example}
\label{ex:bigAlgiTriangleFree}
Suppose $n$ is divisible by $2s$ and take $G_n=\bigsqcup^{n/2s}\widetilde{K_{s,s}}$. By \eqref{eq:abettiprod} and Lemma \ref{lem:abettibip} we have
$$\abetti(G_n)=((4^{s-1}(s-4)+2\cdot 3^s-2^{s+1}+2)^{1/2s})^n.$$
The maximal value of this expression is attained when $s=18$. For $n$ divisible by $36$ we thus have a bipartite graph $G_n$ with $$\abetti(G_n)\approx 2.070^n.$$
\end{example}

\medskip
Next comes the example which we believe to be extremal among complexes with no missing $d$-faces (see Section~\ref{section:nomissing}).
\begin{example}
\label{ex:missingfaces}
Suppose $d\geq 2$ and let $n$ be divisible by $2d+1$. Define
$$K(n,d)=\underbrace{\Delta[2d]^{(d-2)}\ast\cdots\ast\Delta[2d]^{(d-2)}}_{n/(2d+1)}$$
where $\Delta[k]^{(s)}$ is the $s$-dimensional skeleton of the $k$-dimensional simplex. Clearly $\Delta[2d]^{(d-2)}$ is in $\nicef_d$ and so is $K(n,d)$ as the class $\nicef_d$ is closed under joins. Since $\Delta[k]^{(s)}$ is up to homotopy a wedge of ${k\choose s+1}$ spheres, we get
$$\betti(K(n,d))=\Bigg({2d\choose d-1}^{\frac{1}{2d+1}}\Bigg)^n.$$
For $d=2$ we get the same complex we had in Example \ref{ex:bigBetti}. More precisely, if $n$ is divisible by $5$ then $K(n,2)=\ind(\bigsqcup^{n/5}K_5)$.
\end{example}

\medskip
Finally we have an example of a graph whose neighbourhood complex has the total Betti number exponential in the number of vertices.

\begin{example}
\label{ex:neibcomplex}
Suppose $n$ is divisible by $4$ and let $G_n=\bigoplus^{n/4}(2K_2)$. The complex $\neibcomplex(2K_2)$ consists of $4$ isolated points, hence $\betti(\neibcomplex(2K_2))=3$. By Lemma~\ref{lem:neibjoin} we have
$$\betti(\neibcomplex(G_n))=3^{n/4}=(3^{1/4})^n\approx 1.316^n.$$
\end{example}

\section{Final remarks}
\label{section:final}
The methods behind Theorem~\ref{thm:maintheorem3} can be used in a variety of situations. As another example, consider the dominance complex $D(G)$ of a graph $G$, whose faces are the complements of dominating sets of $G$, see for instance \cite{Kaw3,MarTes2}. Then then minimal non-faces of $D(G)$ are of the form $N_G[v]$ for $v\in V(G)$, so Theorem~\ref{thm:maintheorem3}.c) applies with $m=n$.

For this reason it would be interesting to find the optimal value of $\Gamma$ in Theorem~\ref{thm:maintheorem1}, since its improvement leads automatically to improvements in Theorem~\ref{thm:maintheorem3}. Moreover, the result of Theorem~\ref{thm:maintheorem3}.b) about the complex $\neibcomplex(G)$ can possibly be improved even further, as the bipartite graph $\bipa(\neibcomplex(G))$ appearing in the proof is of a rather special form.

For the intermediate classes of complexes with no missing $d$-faces we have the following conjecture.
\begin{conjecture}
\label{conj:nomissingbest}
For every simplicial complex $K\in\nicef(n,d)$ we have $\betti(K)\leq \big({2d\choose d-1}^{\frac{1}{2d+1}}\big)^n$.
\end{conjecture}
When $d=2$ it is the statement of Theorem~\ref{thm:maintheorem1} and it is open for $d\geq 3$.

\medskip
Finally, motivated by the situation in Example~\ref{ex:bigBettiBip}, we make the following conjecture.
\begin{conjecture}
\label{conj:biptri}
If $G$ maximizes $\tbetti(G)$ among $n$-vertex  triangle-free graphs then $G$ is bipartite.
\end{conjecture}

\subsection*{Acknowledgement} Thanks to Marek Kr\v{c}\'al for asking about the possible size of homology of the neighbourhood complex and to Eran Nevo for discussions.


\end{document}